\documentclass[11pt]{article} 
\usepackage[utf8]{inputenc} 
\usepackage{latexsym,amsmath,amssymb,textcomp,amsthm}

\usepackage{geometry} 
\usepackage{tikz}
\usepackage{graphicx} 

\usepackage{calrsfs}
\usepackage[english]{babel}
\usepackage{booktabs} 
\usepackage{array} 
\usepackage{paralist} 
\usepackage{verbatim} 
\usepackage{subfig} 
\usepackage{pgf}
\usepackage{fancyhdr} 
\usepackage{hyperref}
\usepackage{sectsty}
\allsectionsfont{\sffamily\mdseries\upshape} 
\usepackage{amsmath,amsfonts,amssymb}

\usepackage[nottoc,notlof,notlot]{tocbibind} 
\usepackage[titles,subfigure]{tocloft} 

\usetikzlibrary{arrows,automata}

\newtheorem{prop}{Proposition}[subsection]
\newtheorem{lem}[prop]{Lemma}
\newtheorem{theo}{Theorem}
\newtheorem{theo2}{Theorem}
\newtheorem{defin}{Definition}

\newcommand{\cg}{[\kern-0.15em [}
\newcommand{\cd}{]\kern-0.15em]}

\newcommand{\E}{\mathbb{E}}
\newcommand{\Prob}{\mathbb{P}}

\newcommand{\R}{\mathbf{R}}

\newcommand{\N}{\mathbf{N}}

\newcommand{\dd}{\mathrm{d}}

\title{Random walk in random environment and their time-reversed counterpart }
\author{Rémy Poudevigne--Auboiron}
\date{} 

\begin{document}
\maketitle
\begin{abstract}
The random walk in Dirichlet environment is a random walk in random environment where the transition probabilities are independent Dirichlet random variables. This random walk exhibits a property of statistical invariance by time-reversal which leads to several results. More precisely, a time-reversed random walk in Dirichlet environment (with null divergence) is also a random walk in random environment where the transition probabilities are independent Dirichlet random variables with different parameters. We show that on all graphs that satisfy a few weak assumptions, a random walk in random environment with independent transition probabilities and such that the transition probabilities of the time-reversed random walk in random environment are also independent is a random walk in Dirichlet environment. 
\end{abstract}
\section{Introduction and results}
\subsection{Introduction}

In this paper we deal with random walks in random environments (RWRE) with independent transition probability. This model has been studied for several decades \cite{zeitouni}: since the 80s in dimension 1 and for two decades in higher dimensions. In one dimension RWRE exhibit a key property: reversibility. Thanks to this property, the one dimensional case is now well understood (see Solomon $\cite{Solomon}$,Kesten, Kozlov, Spitzer $\cite{Kesten}$ and Sina\"{\i} $\cite{Sinai}$).\\
Unfortunately, in higher dimensions RWRE with iid transition are no longer reversible. For this reason, RWRE in higher dimension is not as well understood as the one dimensional case but important progress has been made. For instance, under some assumptions (see \cite{SznitmanT},\cite{SznitmanEffective},\cite{RamiEquiv},\cite{BergEffectPol},\cite{ElliptCriteria}) regarding the directional transience of the walk and ellipticity, ballisticity and annealed CLT have been proved. Some quenched CLTs have also been proved under stronger assumptions (\cite{SznitmanCLT},\cite{SznitmanZerner},\cite{SeppaTCL},\cite{BergerZeitouniTCL}). Another direction taken was to look at small perturbation of the simple walk (\cite{BricKupia},\cite{SznitZeitCLTBrownian},\cite{ExitBall0},\cite{LowDis1},\cite{LowDis2}). \\
In this paper we look at a specific case of RWRE: random walks in Dirichlet environment. That is to say random walk where the transition probabilities at each site are iid and have a Dirichlet distribution. It was first introduced because of its link to the linearly directed-edge reinforced random walk (\cite{Pemantle}). This model (under an additional property of null-divergence) exhibits a property of invariance after time-reversal, that is to say that the time reversed random walk is also a random walk in Dirichlet environment (\cite{transdim},\cite{transdir1}). This property makes some calculations explicit which allows to find some non-trivial results. For instance, it was shown that in dimension $d\geq 3$, the walk is transient (\cite{transdim}) and there is an invariant distribution for the process seen from the point of view of the particle (\cite{DirEnvirPOVPart}). It is also known for which parameters the walk is directionaly transient (\cite{transdir1},\cite{transdir2},\cite{BouchetSubbal}) and in this case the walk is either ballistic or converges to a stable Levy process (\cite{MoiSubbal}). These are still open questions in the general case. \\
A natural question is whether the random walk in Dirichlet environment is the only RWRE with independent transition probability that exhibits the time-reversal property. We show that, indeed, under some weak conditions on the graph considered, if both the environment and the time-reversed environment have transition probabilities independent at each site, then the transition probabilities follow a Dirichlet distribution (with null-divergence).

\subsection{Definition and statement of the results}

We will first need a few definitions before we can properly state the result. We will look at random environments on directed graphs.
\begin{defin}
An environment $\omega$ on an oriented graph $G=(V,E)$ is a function from the set of edges $E$ to $[0,1]$ such that for any vertex $x\in E$ we have:
\[
\sum\limits_{y,(x,y)\in E} \omega(x,y)=1.
\]
For any oriented graph $G$ let $\Omega^G$ be the set of all environments on $G$.
\end{defin}
The goal of this paper is to study the time-reversed walk. It is obtained by reversing the graph and the environment on this graph. To any oriented graph and any environment on this graph the associated reversed graph and environment are defined as follows.
\begin{defin}
For any graph $(V,E)$, its reversed graph $(\tilde{V},\tilde{E})$ is obtained by keeping all the vertices and flipping all the edges ie: $\tilde{V}=V$ and $\tilde{E}=\left\{ (x,y), (y,x)\in E)\right\}$
\end{defin}
\begin{defin}
Let $(V,E)$ be a graph and $\omega$ an environment on this graph. The reversed environment $\tilde{\omega}$ on the reversed graph $(\tilde{V},\tilde{E})$ is defined by $\tilde{\omega}(x,y)=\omega(y,x)\frac{\pi_y}{\pi_x}$ where $\pi$ is the stationary distribution (i.e for any vertex $x$, $\pi_x=\sum \omega(y,x) \pi_y$).
\end{defin}
For a given environment, it is not easy to compute its reversed environment. As a consequence, for a given graph and a given law on its environments, it is in general hard to compute the law of the reversed environment. However, in the specific case of iid sites with Dirichlet distribution and null divergence, the computation becomes easy. We will now explain what the Dirichlet distribution is. For any family $(\alpha_1,\dots,\alpha_n)$ of positive weights, Dirichlet random variables of parameter $\alpha:=(\alpha_1,\dots,\alpha_{2d})$ have the following density:
\[
\frac{\Gamma\left(\sum\limits_{i=1}^{2d}\alpha_i\right)}{\prod\limits_{i=1}^{2d}\Gamma(\alpha_i)}\left(\prod\limits_{i=1}^{2d}x_i^{\alpha_i-1}\right)\dd x_1\dots \dd x_{2d-1}
\]
on the simplex
\[
\{(x_1,\dots,x_{2d})\in (0,1]^{2d},\sum\limits_{i=1}^{2d}x_i=1\}.
\]
Let $G=(V,E)$ be a directed graph. Let $\alpha:=(\alpha_e)_{e\in E}$ be a family of positive weights. Now, let $\Prob^{G,\alpha}$ be the law on $\Omega$ the environments at each vertex are independent. For any vertex $x\in V$, the law of the family $(\omega(x,y))_{y,(x,y)\in E}$ under $\Prob^{G,\alpha}$ is a Dirichlet law of parameters $(\alpha_{(x,y)})_{y,(x,y)\in E}$. It was proved in $\cite{transdim}$ that if the divergence of $\alpha$ is null:
\[
\forall x\in V,\ \ \sum\limits_{y,(x,y)\in E}\alpha_{(x,y)}=\sum\limits_{y,(y,x)\in E}\alpha_{(y,x)}
\]
then the law of the reversed environment is $\Prob^{\tilde{G},\check{\alpha}}$ with $\check{\alpha}_{(x,y)}=\alpha_{(y,x)}$. This property makes many calculations explicit. We want to see if it is possible to find other law on the set of environments where the transition probabilities are iid at each site and the law of the reversed environment is also iid on each site. We show that under weak assumptions on the graph considered, no such other non-deterministic law exists. To precisely state this, we first need a couple of definitions regarding graphs.  
\begin{defin}
A directed graph $(V,E)$ is strongly connected if for any pair of vertices $(x,y)\in V^2$ there exists a path along the oriented edges of the graph that goes from $x$ to $y$. 
\end{defin}
\begin{defin}
A directed graph $(V,E)$ is 2-connected if it is strongly connected and if for any vertex $x\in V$, the graph obtained by removing $x$ and all the edges coming from $x$ or going to $x$ is still strongly connected.
\end{defin}
Now we can finally state the theorem.
\begin{theo}
Let $(V,E)$ be a finite directed graph and $\omega$ transition probabilities on this graph that satisfy the following properties:
\begin{itemize}
\item the graph has no multiple edges,
\item the graph and the reversed graph are 2-connected,
\item the transition probabilities are of positive expectation,
\item the transition probabilities are independent.
\end{itemize}
If the transition probabilities of the reversed environment are independent then, the transition probabilities are independent Dirichlet random variables with null divergence or are deterministic. 
\end{theo}

\section{The proof}
\subsection{A few general results on graphs}

We will use the following notations for simplification.
\begin{defin}
Let $G=(V,E)$ be a directed graph, for every vertex $x\in V$ we define $E_x$ as the set of all edges starting at $x$ and $E^x$ as the set of all edges ending at $x$. And we also define: $V_x=\{y\in V, (x,y)\in E_x\}$ and $V^x=\{y\in V, (y,x)\in E^x\}$.
\end{defin}
In almost all the proofs we will use paths and cycle on graphs. They are defined as follows.
\begin{defin}
A path $\gamma$ on a graph $(V,E)$ is a sequence of vertices $(\gamma_1,\dots,\gamma_n)$ such that for any $i< n, (\gamma_i,\gamma_{i+1})\in E$.\\
The length of $|\gamma|$ of the path $\gamma$ is equal to $n$. \\
We will write $(x,y)\in \gamma$ to say that there exists an integer $i\in [\![1,n-1]\!]$ such that $(x,y)=(\gamma_i,\gamma_{i+1})$.\\
A cycle $\sigma$ is a path such that $\sigma_1=\sigma_{|\sigma|}$. 
\end{defin}
We will now prove a couple of results on oriented graphs that will help us for the proof of the main result. 
\begin{lem}\label{lem:2vertex}
If (V,E) be a 2-connected graph, with no multiple edges and at least three vertices, then for every $x\in V$, $E_x$ contains at least two edges.
\end{lem}
\begin{proof}
Let $x$ be a vertex in $V$, then $E_x$ contains at least one edge because otherwise there would be no non-trivial path starting at $x$ and the graph would not be strongly connected. Now, if $E_x$ had only one edge $(x,y)$ then for the same reasons, the graph obtained by removing the vertex $y$ would not be strongly connected. Therefore, $E_x$ contains at least two elements
\end{proof}
\begin{lem}\label{lem:2path}
Let (V,E) be a 2-connected graph, with no multiple edges. Let $x_1$,$x_2$ and $y$ be three distinct vertices in $V$. There exists two simple paths $\gamma^1$ and $\gamma^2$ going from $x_1$ to $y$ and from $x_2$ to $y$ respectively such that the only point at which they intersect is $y$ (ie $\gamma^1_i=\gamma^2_j$ iff $\gamma^1_i=y$).
\end{lem}
\begin{proof}
First, for any $z_1,z_2\in V$, let $\Gamma(z_1,z_2)$ be the set of simple paths that go from $z_1$ to $z_2$ along the directed edges of $E$. Now for any three distinct points $x_1,x_2,z\in V$, we look at the set of paths: 
\[
\tilde{\Gamma}^{\prime}(x_1,x_2,y):=\bigcup\limits_{z\in V}\Gamma(x_1,z)\times\Gamma(x_2,z)\times\Gamma(z,y).
\]
Let $\tilde{\Gamma}$ be the subset of $\tilde{\Gamma}^{\prime}$ that only contains triplets $\gamma^1,\gamma^2,\tilde{\gamma}$ such that the only point at which at least two of the paths intersect is the endpoint of $\gamma^1$ and $\gamma^2$ which is also the starting point of $\tilde{\gamma}$. \\
We first want to show that $\tilde{\Gamma}$ is not empty. We take two simple paths $\gamma^1$ and $\gamma^2$ going from $x_1$ to $y$ and from $x_2$ to $y$ respectively. If they do not intersect except at $y$ then we have proved the lemma, otherwise we define:
\[
\tau_2:=\inf\left\{i\geq 0,\exists j\geq 0, \gamma^1_j= \gamma^2_i \right\}<\infty,
\]
and
\[
\tau_1:=\inf\left\{i\geq 0,\gamma^1_i= \gamma^2_{\tau_2} \right\}<\infty.
\]
The triplet $((\gamma^1_0, \dots, \gamma^1_{\tau_1}),(\gamma^2_0, \dots, \gamma^2_{\tau_2}),(\gamma^1_{\tau_1}, \dots, \gamma^1_{|\gamma^1|}))$ is in $\tilde{\Gamma}$ by definition of $\tau_1$ and $\tau_2$. Therefore $\tilde{\Gamma}$ is not empty.
Now we look at a triplet $(\gamma^1,\gamma^2,\tilde{\gamma})\in \tilde{\Gamma}$ that minimizes the length $|\tilde{\gamma}|$. If $\tilde{\gamma}$ is just a point then we have the result we want. Otherwise we show that we can shorten the length of $\tilde{\gamma}$. First, let $z$ be the point at which the three paths of the triplet intersect. Now let $p$ be a path that goes from $x_1$ to $y$ without going through $z$. Let $\tau_1= \sup \{ i, \exists j\in\N, p_i = \gamma^1_j \text{ or } p_i = \gamma^2_j\}$ and $\tau_2= \inf \{i>\tau_1, \exists j\in\N, p_i=\tilde{\gamma}_j\}$. We will assume that $p_{\tau_1}$ is a point of $\gamma_1$ (the proof is exactly the same if it is a point of $\gamma_2$). We also define $\tilde{\tau}_1=\inf \{i\in\N,\gamma^1_i=p_{\tau_1}\}$ and $\tilde{\tau}_2=\sup \{i\in\N,\tilde{\gamma}_i=p_{\tau_2}\}$. The following triplet is an element of $\tilde{\Gamma}$ with a smaller length for the third element compared to $(\gamma^1,\gamma^2,\tilde{\gamma})$:
\[
((\gamma^1_0,\dots,\gamma^1_{\tilde{\tau}_1},p_{\tau_1+1},\dots,p_{\tau_2}),
(\gamma^2_0,\dots,\gamma^2_{|\gamma^2|},\tilde{\gamma}_1,\dots,\tilde{\gamma}_{\tilde{\tau}_2}),
(\tilde{\gamma}_{\tilde{\tau}_2},\dots,\tilde{\gamma}_{\tilde{\gamma}}))
\]
and therefore we have the desired result.
\end{proof}

\subsection{Moment functions}

\begin{defin}
Let $G=(V,E)$ be a directed graph. A function $N: E \mapsto \R$ is said to be of null divergence if: \\
\[
\forall x\in V, \sum\limits_{e\in E_x} N(e)= \sum\limits_{e\in E^x} N(e).
\]
\end{defin}
\begin{defin}
A moment function $f$ of a graph $(V,E)$ will be a set of functions $\{f_x,x\in V \}$ such that for every vertex $x\in V$, $f_x$ is a function from $\N^{E_x}$ to $(0,\infty)$. To simplify notation, instead of writing:
\[
f_x(n_{xy_1},n_{xy_2},\dots,n_{xy_r})
\]
we will write:
\[
f_x\left(\sum\limits_{e\in E_x} n_e \vec{e}\right).
\]
\end{defin}
\begin{defin}
Let $G=(V,E)$ be a graph, $f$ a moment function of $G$ and $\check{f}$ be a moment function of the reversed graph $\tilde{G}=(v,\tilde{E})$. $f$ and $\check{f}$ are compatible if for all functions $N: E \mapsto \N$ of null divergence:  \\
\[
\prod\limits_{x\in V}f_x\left(\sum\limits_{e\in E_x} N(e) \vec{e}\right)
=\prod\limits_{x\in V}\check{f}_x\left(\sum\limits_{y\in V^x} N((x,y)) \overrightarrow{yx}\right).
\]
\end{defin}
Let $(V,E)$ be a graph and $\omega$ a random environment on this graph that both satisfy the conditions of the theorem. To prove our theorem, we first want to show that the moment of our transition probabilities are of the form :
\[
\E\left(\prod\limits_{y\in V^x} \omega(x,y)^{n_{(x,y)}}\right)=\frac{\prod\limits_{y\in V^x} h_{(x,y)}(n_{(x,y)})}{\tilde{h}_x\left(\sum\limits_{y\in V^x}n_{(x,y)}\right)},
\]
for some functions $(h_{e})_{e\in E}$ ad $(\tilde{h}_x)_{x\in V}$. In order to do that, we first need to find suitable candidates for the functions $(h_e)_{e\in E}$ and $(\tilde{h}_x)_{x\in V}$. The following lemma will help us do that (in the proof of the theorem where we prove that the moments of the transition probabilities of the environment $\omega$ and its reversed environment are compatible moment functions).
\begin{lem}\label{lem:hfunctions}
Let $G=(V,E)$ be a 2-connected directed graph. Let $f$ be a moment function of $G$ and $\check{f}$ be a moment function of the reversed graph $\tilde{G}$ such that $f$ and $\check{f}$ are compatible. We also assume that: \\
\[
\forall x\in V, f_x(0)=\check{f}_x(0)=1.
\]  
Then for every vertex $x\in V$ there exists a function $\tilde{h}_x: \N \mapsto (0,\infty)$ and for every edge $e \in E$ there exists a function $h_e:  \N \mapsto (0,\infty)$ such that:
\[
\begin{aligned}
&\forall x\in V, \forall y\in V_x, \forall n\in \N, f_x(n\overrightarrow{xy})=\frac{h_{(x,y)}(n)}{\tilde{h}_x(n)} \text{ and }\\
&\forall x\in V, \forall y\in V^x, \forall n\in \N, \check{f}_x(n\overrightarrow{xy})=\frac{h_{(y,x)}(n)}{\tilde{h}_x(n)}.
\end{aligned}
\]
\end{lem}
\begin{proof}
For $n=0$, the result is obvious, we just need to take  $h_e(0)=1$ for all $e\in E$ and $ \tilde{h}_x(0)=1$ for all $x\in V$. Now we choose $n\geq 1$. For every edge $(x,y)\in E$ we write: \\
\[
g_{(x,y)}(n):=\frac{f_x(n\overrightarrow{xy})}{\check{f}_y(n \overrightarrow{yx})}.
\]
For a simple cycle (i.e a cycle that never visits a point more than once, except for the first point because it is also the last one) $\sigma$, the compatibility of $f$ and $\check{f}$ tells us that:
\[
\prod\limits_{i} f_{\sigma(i)}(n \overrightarrow{\sigma(i)\sigma(i+1)})=\prod\limits_{i} \check{f}_{\sigma(i)}(n \overrightarrow{\sigma(i)\sigma(i-1)}),
\]
which means:
\[
\prod\limits_{i} g_{(\sigma(i),\sigma(i+1))}(n)=1.
\]
Since any cycle is the union of simple cycles, the above property is true for any cycle. Now, we choose a vertex $x\in V$ and we set $\tilde{h}_x(n):=1$. Let $y$ be a vertex in $V$, $\gamma^1$ and $\gamma^2$ two paths from $x$ to $y$ and $\tilde{\gamma}$ a path from $y$ to $x$. We have:
\[
\prod\limits_{i} g_{(\gamma^1_i,\gamma^1_{i+1})}(n)\prod\limits_{i} g_{(\tilde{\gamma}_i,\tilde{\gamma}_{i+1})}(n)=1,
\]
and
\[
\prod\limits_{i} g_{(\gamma^2_i,\gamma^2_{i+1})}(n)\prod\limits_{i} g_{(\tilde{\gamma}_i,\tilde{\gamma}_{i+1})}(n)=1.
\]
Therefore,
\[
\prod\limits_{i} g_{(\gamma^1_i,\gamma^1_{i+1})}(n)=\prod\limits_{i} g_{(\gamma^2_i,\gamma^2_{i+1})}(n).
\]
This equality allows us to define $\tilde{h}_y(n)$ by:
\[
\tilde{h}_y(n):=\prod\limits_{i} g_{(\gamma^1_i,\gamma^1_{i+1})}(n),
\]
because it does not depend on the path we chose to get to $y$. Now let $y_1$ and $y_2$ be two vertices such that $(y_1,y_2)$ is an edge. Let $\gamma$ be a path that goes from $x$ to $y_1$ and $\tilde{\gamma}$ be the same path to which we add the edge $(y_1,y_2)$ at the end, we have:
\[
g_{(y_1,y_2)}(n)=\frac{g_{(y_1,y_2)}(n)\prod g_{(\gamma(i),\gamma(i+1))}(n)}{\prod g_{(\gamma(i),\gamma(i+1))}(n)}
=\frac{\prod g_{(\tilde{\gamma}(i),\tilde{\gamma}(i+1))}(n)}{\prod g_{(\gamma(i),\gamma(i+1))}(n)}
=\frac{\tilde{h}_{y_2}(n)}{\tilde{h}_{y_1}(n)}.
\]
And now, if we choose $h_{(y_1,y_2)}(n)$ such that:
\[
f_{y_1}(n\overrightarrow{y_1y_2})=\frac{h_{(y_1,y_2)}(n)}{\tilde{h}_{y_1}(n)},
\]
then we also have:
\[
\check{f}_{y_2}(n\overrightarrow{y_2y_1})=\frac{h_{(y_1,y_2)}(n)}{\tilde{h}_{y_2}(n)}.
\]
\end{proof}
Unfortunately the functions we have found are not uniquely defined. Indeed, for any $n\in \N^*$ we can multiply all values $\left(h_e(n)\right)_{e\in E}$ and $(\tilde{h}_x)_{x\in V}$ by a constant $\Delta(n)\not =0$ and all the properties of the previous lemma would stay true. The goal of the next lemma is to show that it is the only change we can make to the previous functions. This will be used, in addition to the previous lemma, to show (in the proof of the theorem) that the moments of transition probabilities $\omega$ that satisfy the conditions of the theorem are of the form:
\[
\E\left(\prod\limits_{y\in V^x} \omega(x,y)^{n_{(x,y)}}\right)=\frac{\prod\limits_{y\in V^x} h_{(x,y)}(n_{(x,y)})}{\tilde{h}_x\left(\sum\limits_{y\in V^x}n_{(x,y)}\right)},
\]
for some functions $(\tilde{h}_x)_{x\in V}$ and $(h_e)_{e\in E}$.
\begin{lem}\label{lem:deltacorrection}
Let $G=(V,E)$ be a 2-connected graph and $f$ and $\check{f}$ two compatible moment functions such that:
\[
\forall (x,y)\in E, \forall n\in N, f_x(n\overrightarrow{xy})=\check{f}_y(n\overrightarrow{yx})=1.
\]
There exists $\Delta: \N \mapsto (0,\infty)$ with $\Delta(0)=\Delta(1)=1$ such that:
\[
\forall x\in V, f_x\left(\sum\limits_{y\in V_x} N((x,y))\overrightarrow{xy}\right)=\frac{\prod\limits_{y\in V_x} \Delta(N((x,y)))}{\Delta\left(\sum\limits_{y\in V_x} N((x,y))\right)}
\]
and,
\[
\forall x\in V, \check{f}_x\left(\sum\limits_{y\in V^x} N((y,x))\overrightarrow{xy}\right)=\frac{\prod\limits_{y\in V^x} \Delta(N((y,x)))}{\Delta\left(\sum\limits_{y\in V^x} N((y,x))\right)}.
\]
\end{lem}
\begin{proof}
We will construct $\Delta$ by induction. There will be two parts to the proof, first the existence of $\Delta(2)$ and then the existence of $\Delta(i)$ for $i\geq 3$.\\
First we prove the existence of $\Delta(2)$. Let $x\in V$ be a vertex, we want to show that there exists $\Delta_x(2)$ such that:
\[
\begin{aligned}
&\forall y_1,y_2\in V_x, y_1\not=y_2 \implies f_x(\overrightarrow{xy_1}+\overrightarrow{xy_2})=\frac{\Delta(1)^2}{\Delta_x(2)}= \frac{1}{\Delta_x(2)} \\
\text{and } &\forall y_1,y_2\in V^x, y_1\not=y_2 \implies \check{f}_x(\overrightarrow{xy_1}+\overrightarrow{xy_2})= \frac{1}{\Delta_x(2)}.
\end{aligned}
\]
If both $V_x$ and $V^x$ have only two elements then we call the two elements of $V_x$: $y_1$ and $y_2$. By lemma \ref{lem:2path} there exists two paths $\gamma^1$ and $\gamma^2$ that go respectively from $y_1$ to $x$ and from $y_2$ to $x$ and that only intersect in $x$. We call $z_1$ and $z_2$ the vertices such that the last edges crossed by $\gamma^1$ and $\gamma^2$ are $(z_1,x)$ and $(z_2,x)$ respectively. We call $\sigma^1$ and $\sigma^2$ the simple cycles such that $\sigma^1$ goes through $(x,y_1)$ and then follows the path $\gamma^1$ and $\sigma^2$ goes through $(x,y_2)$ and then follows the path $\gamma^2$. By definition of $\gamma^1$ and $\gamma^2$, the cycles $\sigma^1$ and $\sigma^2$ only intersect in $x$, so every vertex other than $x$ is visited by at most  only one of the two cycles (and only once because they are simple cycles), in particular $z_1\not = z_2$. We have that:
\[
\begin{aligned}
&f_x(\overrightarrow{xy_1}+\overrightarrow{xy_2})
\prod\limits_{v\in V\backslash\{x\}} f_v\left(\sum\limits_{u\in V_v}(1_{(v,u)\in \sigma^1}+1_{(v,u)\in \sigma^2})\overrightarrow{vu}\right)\\
=&\check{f}_x(\overrightarrow{xz_1}+\overrightarrow{xz_2})
\prod\limits_{v\in V\backslash\{x\}} \check{f}_v\left(\sum\limits_{u\in V^v}(1_{(u,v)\in \sigma^1}+1_{(u,v)\in \sigma^2})\overrightarrow{vu}\right).
\end{aligned}
\]
This means that
\[
f_x(\overrightarrow{xy_1}+\overrightarrow{xy_2})=\check{f}_x(\overrightarrow{xz_1}+\overrightarrow{xz_2}).
\]
Therefore $\Delta_x(2)$ exists.\\  
	Otherwise we can assume, without loss of generality, that $V_x$ has at least three elements (if it is not the case, then $V^x$ has at least three elements by lemma \ref{lem:2vertex} and we just have to consider the reversed graph). For the sake of clarity, the figure below describes what is written in this paragraph. Let $y_1,y_2\text{ and }y_3\in V_x$ be three distinct vertices. Let $\gamma^1$ and $\gamma^2$ be two simple paths that go from $y_1$ to $x$ and from $y_2$ to $x$ respectively and that only intersect in $x$. We call $z_1$ the vertex such that $(z_1,x)$ is the last edge crossed by $\gamma^1$ and $z_2$ the vertex such that $(z_2,x)$ is the last edge crossed by $\gamma^2$. Now let $\gamma^3$ be a path from $y_3$ to $z_1$ or $z_2$ that does not go through $x$. Let $\tau_1$ be the first time such that $\gamma^3(\tau_1)$ is either a vertex of $\gamma^1$ or $\gamma^2$ and let $A$ be the corresponding vertex, $A:=\gamma^3(\tau_1)$. We can assume that $A$ is a vertex of $\gamma^1$ (if it is not the case we can just exchange the role of $y_1$ and $y_2$). Let $\tau_2$ be the time such that $\gamma^1_{\tau_2}=\gamma^3_{\tau_1}$. Let $\tilde{\gamma}^3$ be the path that starts at $y_3$ then follows $\gamma^3$ up to $\gamma^3_{\tau_1}$ and then follows $\gamma^1$ from $\gamma^1_{\tau_2}$ to $x$. In particular, the last edge crossed by $\tilde{\gamma}^3$ is $(z_1,x)$. 
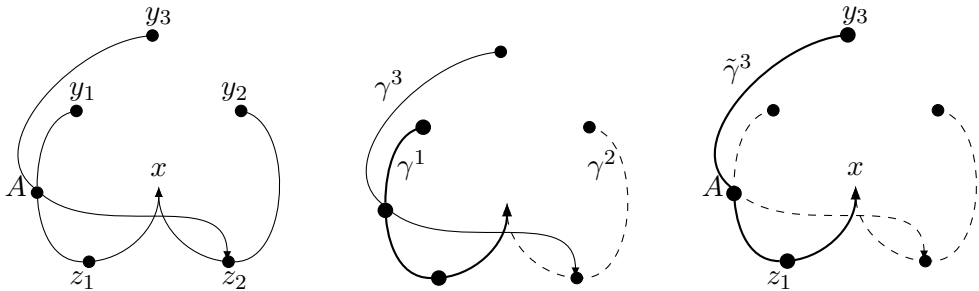
\[
\begin{matrix}
\begin{tikzpicture}
\draw[*-,>=latex] (-1,1)node[above]{$y_1$} to[out=180,in=90] (-1.6,0)node[left]{$A$};
\draw[*-,>=latex] (-1.6,0) to[out=-90,in=180] (-1,-1)node[below]{$z_1$};
\draw[*->,>=latex] (-1,-1) to[out=0,in=-90] (0,0)node[above]{$x$};
\draw[*-,>=latex] (1,1)node[above]{$y_2$} to[out=0,in=0] (1,-1)node[below]{$z_2$};
\draw[*->,>=latex] (1,-1) to[out=180,in=-90] (0,0);
\draw[*-,>=latex] (0,2)node[above]{$y_3$} to[out=180,in=145] (-1.6,-0.05);
\draw[->,>=latex] (-1.6,-0.05) to[out=-45,in=90] (0.9,-1);
\end{tikzpicture}
&\begin{tikzpicture}
\draw[*-,>=latex,thick] (-1,1)node[above]{} to[out=180,in=90] (-1.6,0)node[left]{};
\node at (-1.25,0.5) {$\gamma^1$} ;
\node at (1.25,0.5) {$\gamma^2$} ;
\draw[*-,>=latex,thick] (-1.6,0) to[out=-90,in=180] (-1,-1)node[below]{};
\draw[*->,>=latex,thick] (-1,-1) to[out=0,in=-90] (0,0)node[above]{};
\draw[*-,>=latex,dashed] (1,1)node[above]{} to[out=0,in=0] (1,-1)node[below]{};
\draw[*->,>=latex,dashed] (1,-1) to[out=180,in=-90] (0,0);
\draw[*-,>=latex] (0,2)node[above]{} to[out=180,in=145] (-1.6,-0.05);
\node at (-1.55,1.55) {$\gamma^3$};
\draw[->,>=latex] (-1.6,-0.05) to[out=-45,in=90] (0.9,-1);
\end{tikzpicture}
&\begin{tikzpicture}
\draw[*-,>=latex,dashed] (-1,1)node[above]{} to[out=180,in=90] (-1.6,0)node[left]{$A$};
\draw[*-,>=latex,thick] (-1.6,0) to[out=-90,in=180] (-1,-1)node[below]{$z_1$};
\draw[*->,>=latex,thick] (-1,-1) to[out=0,in=-90] (0,0)node[above]{$x$};
\draw[*-,>=latex,dashed] (1,1)node[above]{} to[out=0,in=0] (1,-1)node[below]{};
\draw[*->,>=latex,dashed] (1,-1) to[out=180,in=-90] (0,0);
\draw[*-,>=latex,thick] (0,2)node[above]{$y_3$} to[out=180,in=145] (-1.6,-0.05);
\node at (-1.55,1.55) {$\tilde{\gamma}^3$};
\draw[->,>=latex,dashed] (-1.6,-0.05) to[out=-45,in=90] (0.9,-1);
\end{tikzpicture} 
\end{matrix}
\]	
Now we have:
\[
\begin{aligned}
&f_x(\overrightarrow{xy_1}+\overrightarrow{xy_2})\prod\limits_{v\in V\backslash \{ x\}}f_v\left(\sum\limits_{u\in V_v} (1_{(v,u)\in \gamma^1}+1_{(v,u)\in \gamma^2})\overrightarrow{vu}\right)\\
=&\check{f}_x(\overrightarrow{xz_1}+\overrightarrow{xz_2})\prod\limits_{v\in V\backslash \{ x\}}\check{f}_v\left(\sum\limits_{u\in V^v} (1_{(u,v)\in \gamma^1}+1_{(u,v)\in \gamma^2})\overrightarrow{vu}\right)
\end{aligned}
\]
Now since every vertex besides $x$ is visited at most once by $\gamma^1$ and $\gamma^2$ (since these two simple paths only intersect in $x$) we get:
\[
f_x(\overrightarrow{xy_1}+\overrightarrow{xy_2})=\check{f}_x(\overrightarrow{xz_1}+\overrightarrow{xz_2})
\]  
by the same arguments, using $\gamma^2$ and $\tilde{\gamma}^3$ we get 
\[
f_x(\overrightarrow{xy_3}+\overrightarrow{xy_2})=\check{f}_x(\overrightarrow{xz_1}+\overrightarrow{xz_2}).
\]
So 
\[
f_x(\overrightarrow{xy_1}+\overrightarrow{xy_2})=f_x(\overrightarrow{xy_3}+\overrightarrow{xy_2})
\]
Now, if we do exactly the same thing except we switch $y_2$ and $y_3$. We get that either 
\[
f_x(\overrightarrow{xy_1}+\overrightarrow{xy_3})=f_x(\overrightarrow{xy_2}+\overrightarrow{xy_1})
\]
or
\[
f_x(\overrightarrow{xy_1}+\overrightarrow{xy_3})=f_x(\overrightarrow{xy_2}+\overrightarrow{xy_3}).
\]
Either way, we get:
\[
f_x(\overrightarrow{xy_1}+\overrightarrow{xy_3})=
f_x(\overrightarrow{xy_1}+\overrightarrow{xy_2})=
f_x(\overrightarrow{xy_2}+\overrightarrow{xy_3})=
\check{f}_x(\overrightarrow{xz_1}+\overrightarrow{xz_2}).
\]
Now either $V^x$ has three elements or more in which case the same arguments hold for $\check{f}$ or it has only two elements and we only have to consider $\check{f}_x(\overrightarrow{xz_2}+\overrightarrow{xz_1})$, in both case we know that $f_x(\overrightarrow{xy_1}+\overrightarrow{xy_2})=\check{f}_x(\overrightarrow{xz_2}+\overrightarrow{xz_1})$ so $\Delta_x(2)$ exists.\\
Now we have to prove that $\Delta_x(2)$ does not depend on $x$. Let $x$ and $y$ be two points such that $(x,y)\in E$. We want to prove that there exists two simple cycles that both contain the edge $(x,y)$ and that only intersect in $x$ and $y$. It is clearly equivalent to prove that there exists two simple paths that begin at $y$ and end at $x$ and that only intersect at $x$ and $y$. To prove that, let $z\in V^x$ be a vertex, we know that there exists two simple paths $p^1$ and $p^2$ that go respectively from $y$ to $x$ and from $z$ to $x$ and that only intersect at $x$, now we look at the two simple cycles defined as follows: $\sigma^1$ is the cycle that starts at $x$ then goes through $(x,y)$ and then follows $p^1$ back to $x$ and $\sigma^2$ is the cycle that starts at $x$ then goes through $(x,y)$ and $(y,z)$ and then follows $p^2$ back to $x$. This gives us:
\[
\prod\limits_{v\in V}f_v\left(\sum\limits_{u\in V_v}\left(1_{(v,u)\in\sigma^1}+1_{(v,u)\in\sigma^2}\right)\overrightarrow{vu}\right)=\prod\limits_{v\in V}\check{f}_v\left(\sum\limits_{u\in V^v}\left(1_{(u,v)\in\sigma_1}+1_{(u,v)\in\sigma_2}\right)\overrightarrow{vu}\right).
\]
Since for every $v\in V\backslash\{x,y\}$, $v$ cannot be visited by both $\sigma^1$ and $\sigma^2$,  
we get that for every $v\in V\backslash\{x,y\}$:
\[
\begin{aligned}
&f_v\left(\sum\limits_{u\in V_v}\left(1_{(v,u)\in\sigma^1}+1_{(v,u)\in\sigma^2}\right)\overrightarrow{vu}\right)=1, \\
\text{and } &\check{f}_v\left(\sum\limits_{u\in V^v}\left(1_{(u,v)\in\sigma^1}+1_{(u,v)\in\sigma^2}\right)\overrightarrow{vu}\right)=1,
\end{aligned}
\]
we get:
\[
f_x(2\overrightarrow{xy})f_y\left(\sum\limits_{u\in V_y}\left(1_{(y,u)\in\sigma^1}+1_{(y,u)\in\sigma^2}\right)\overrightarrow{yu}\right)
=\check{f}_x\left(\sum\limits_{u\in V^x}\left(1_{(u,x)\in\sigma^1}+1_{(u,x)\in\sigma^2}\right)\overrightarrow{xu}\right)\check{f}_y(2\overrightarrow{yx}).
\]
This in turns means that
\[
f_y\left(\sum\limits_{u\in V_y}\left(1_{(y,u)\in\sigma^1}+1_{(y,u)\in\sigma^2}\right)\overrightarrow{yu}\right)=\check{f}_x\left(\sum\limits_{u\in V^x}\left(1_{(u,x)\in\sigma^1}+1_{(u,x)\in\sigma^2}\right)\overrightarrow{xu}\right).
\]
Therefore
\[
\frac{\Delta(1)^2}{\Delta_x(2)}=\frac{\Delta(1)^2}{\Delta_y(2)}.
\]
So we get $\Delta_x(2)=\Delta_y(2)$ and since the graph is connected, $\Delta(2)$ exists.\\
Now we can prove the existence of $\Delta(i)$ for $i\geq 3$ by induction.\\
First we assume that $\Delta(i)$ exists for $i\leq n$ for some $n\geq 2$. We want to prove that $\Delta(n+1)$ exists. First let $x\in V$ be a point of the graph and set two vertices $y_1,y_2\in V_x$. We know that there exists two simple paths $\gamma^1: y_1 \rightarrow x$ and $\gamma^2: y_2 \rightarrow x$ that only intersect in $x$ by lemma \ref{lem:2path}. We will call $z_1$ and $z_2$ the points such that the last edge through which $\gamma^1$ and $\gamma^2$ go are $(z_1,x)$ and $(z_2,x)$ respectively. We now consider two sequences of points $(y_3 \dots y_{n+1} )\in V_x$ and $(z_3 \dots z_{n+1} )\in V^x$ and a sequence of simple paths $(\gamma^3 \dots \gamma^{n+1} )$ such that for every $i$, $\gamma^i$ goes from $y_i$ to $z_i$ without passing through $x$ and then goes through the edge $(z_i,x)$. Now, for every $i\leq n+1$ we look at the simple cycle $\sigma^i$ which starts at $x$ then goes along the edge $(x,y_i)$, then follows the path $\gamma^i$. By construction of the cycles we have:
\[
\forall v\in V\backslash \{x\}, \#\left\{i,v\in \sigma^i\right\} \leq n.
\]
Therefore we get:
\[
\begin{aligned}
&f_x\left(\sum\limits_{i\leq n+1} \overrightarrow{xy_i}\right)\prod\limits_{v\in V\backslash \{x\}}\frac{\prod\limits_{e\in E_v} \Delta(\#\left\{i,e\in\sigma^i\right\})}{\Delta(\#\left\{i,v\in\sigma^i\right\})} \\
=&\check{f}_x\left(\sum\limits_{i\leq n+1} \overrightarrow{xz_i}\right)\prod\limits_{v\in V\backslash \{x\}}\frac{\prod\limits_{e\in E^v} \Delta(\#\left\{i,e\in\sigma^i\right\})}{\Delta(\#\left\{i,v\in\sigma^i\right\})}.
\end{aligned}
\]
So we get:
\[
\frac{f_x\left(\sum\limits_{i\leq n+1} \overrightarrow{xy_i}\right)}{\prod\limits_{e\in E_x} \Delta(\#\left\{i,(x,y_i)=e\right\})}=\frac{\check{f}_x\left(\sum\limits_{i\leq n+1} \overrightarrow{xz_i}\right)}{\prod\limits_{e\in E^x} \Delta(\#\left\{i,(z_i,x)=e\right\})}.
\]
Now we just need to use this equality to prove that $\frac{f_x\left(\sum\limits_{i\leq n+1} \overrightarrow{xy_i}\right)}{\prod\limits_{e\in E_x} \Delta(\#\left\{i,(x,y_i)=e\right\})}$ does not depend on the sequence $(y_i)$. Since the value of $(z_3,\dots,z_{n+1})$ does not depend on $(y_3,\dots,y_{n+1})$ we have that $\frac{f_x\left(\sum\limits_{i\leq n+1} \overrightarrow{xy_i}\right)}{\prod\limits_{e\in E_x} \Delta(\#\left\{i,(x,y_i)=e\right\})}$ does not depend on $(y_3,\dots,y_{n+1})$. To simplify notations we will write:
\[
g(y_1,\dots,y_{n+1})=\frac{f_x\left(\sum\limits_{i\leq n+1} \overrightarrow{xy_i}\right)}{\prod\limits_{e\in E_x} \Delta(\#\left\{i,(x,y_i)=e\right\})}
\]
Now let $(y^1_1,\dots,y^1_{n+1})$ and $(y^2_1,\dots,y^2_{n+1})$ be two sequences of vertices in $V_x$, we have:
\[
\begin{aligned}
g(y^1_1,\dots,y^1_{n+1}) &= g(y^1_1,y^1_2,y^2_1,\dots,y^2_{n-1}) \\
&=g(y^2_1,y^1_2,y^1_1,y^2_2,\dots,y^2_{n-1}) \\ 
&=g(y^2_1,y^1_2,y^2_2,\dots,y^2_n) \\ 
&=g(y^2_1,y^2_2,y^1_2,y^2_3,\dots,y^2_n) \\
&=g(y^2_1,y^2_2,y^2_3,\dots,y^2_{n+1}) \\
&=g(y^2_1,\dots,y^2_{n+1}).
\end{aligned}
\]
So we have that for every $x\in V$ there is a $\Delta_x(n+1)$ such that:
\[
f_x\left(\sum\limits_{i\leq n+1} \overrightarrow{xy_i}\right)=\frac{\prod\limits_{e\in E_x} \Delta(\#\left\{i,(x,y_i)=e\right\})}{\Delta_x(n+1)}
\]
and
\[
\check{f}_x\left(\sum\limits_{i\leq n+1} \overrightarrow{xz_i}\right)=\frac{\prod\limits_{e\in E^x} \Delta(\#\left\{i,(z_i,x)=e\right\})}{\Delta_x(n+1)}.
\]
Now we just need to prove that this $\Delta_x(n+1)$ does not depend on $x$. Let $x,y\in V$ be two vertices such that $(x,y)$ is an edge, we want to prove that $\Delta_x(n+1)=\Delta_y(n+1)$. This will yield the result we want since the graph is connected. Let $z$ be a point in $V_y\backslash \{ x \}$ which is not empty by lemma \ref{lem:2vertex}. Let $p^1$ and $p^2$ be two simple paths that go respectively from $y$ to $x$ and from $z$ to $x$ and that only intersect in $x$. We will not look at those paths but at the simple cycles $\sigma^1$ and $\sigma^2$ defined as follows: $\sigma^1$ starts at $x$, then goes along the edge $(x,y)$ and finally follows the path $p^1$ back to $x$, $\sigma^2$ starts at $x$, then goes along the edges $(x,y)$ and $(y,z)$ and finally follows the path $p^2$ back to $x$. Those two simple cycles only intersect in $x$ and $y$. This gives us:
\[
\prod\limits_{v\in V}f_v\left(\sum\limits_{u\in V_v}\left(n1_{(v,u)\in\sigma^1}+1_{(v,u)\in\sigma^2}\right)\overrightarrow{vu}\right)=\prod\limits_{v\in V}\check{f}_v\left(\sum\limits_{u\in V^v}\left(n1_{(u,v)\in\sigma_1}+1_{(u,v)\in\sigma_2}\right)\overrightarrow{vu}\right).
\]
Since for every $v\in V\backslash\{x,y\}$, $v$ cannot be visited by both $\sigma^1$ and $\sigma^2$,  
we get that for every $v\in V\backslash\{x,y\}$:
\[
\begin{aligned}
&f_v\left(\sum\limits_{u\in V_v}\left(n1_{(v,u)\in\sigma^1}+1_{(v,u)\in\sigma^2}\right)\overrightarrow{vu}\right)=1, \\
\text{and } &\check{f}_v\left(\sum\limits_{u\in V^v}\left(n1_{(u,v)\in\sigma^1}+1_{(u,v)\in\sigma^2}\right)\overrightarrow{vu}\right)=1,
\end{aligned}
\]
Now, we only have to look at $x$ and $y$. We get:
\[
f_x((n+1)\overrightarrow{xy})f_y\left(\sum\limits_{u\in V_y}\left(n1_{(y,u)\in\sigma^1}+1_{(y,u)\in\sigma^2}\right)\overrightarrow{yu}\right)
=\check{f}_x\left(\sum\limits_{u\in V^x}\left(n1_{(u,x)\in\sigma^1}+1_{(u,x)\in\sigma^2}\right)\overrightarrow{xu}\right)\check{f}_y((n+1)\overrightarrow{yx}).
\]
This in turns means that
\[
f_y\left(\sum\limits_{u\in V_y}\left(n1_{(y,u)\in\sigma^1}+1_{(y,u)\in\sigma^2}\right)\overrightarrow{yu}\right)=\check{f}_x\left(\sum\limits_{u\in V^x}\left(n1_{(u,x)\in\sigma^1}+1_{(u,x)\in\sigma^2}\right)\overrightarrow{xu}\right).
\]
Therefore
\[
\frac{\Delta(n)\Delta(1)}{\Delta_x(n+1)}=\frac{\Delta(n)\Delta(1)}{\Delta_y(n+1)}.
\]
So we get $\Delta_x(n+1)=\Delta_y(n+1)$ and since the graph is connected, $\Delta(n+1)$ exists. Therefore $\Delta$ exists by induction and we have proved the lemma.
\end{proof}
Now, the last thing we need to do is show that if the moments of transition probabilities are of the form
\[
\E\left(\prod\limits_{y\in V^x} \omega(x,y)^{n_{(x,y)}}\right)=\frac{\prod\limits_{y\in V^x} h_({x,y)}(n_{(x,y)})}{\tilde{h}_x\left(\sum\limits_{y\in V^x}n_{(x,y)}\right)},
\]
for some functions $(h_e)_{e\in E}$ and $(\tilde{h}_x)_{x\in V}$, then they follow a Dirichlet law or are deterministic. This can be done by using the following lemma and that for transition probabilities $\omega$, any vertex $x\in V$ and any integers $(n_{(x,y)})_{y\in V_x}$, we have:
\[
\E\left(\prod\limits_{y\in V_x}\omega(x,y)^{n_{(x,y)}}\right)=
\sum\limits_{z\in V_x}\E\left(\prod\limits_{y\in V_x}\omega(x,y)^{n_{(x,y)}+1_{y=z}}\right).
\]
This equality is a direct consequence of this other equality this other equality: $\sum\limits_{y\in V_x}\omega(x,y)=1$.
\begin{lem}\label{lem:momentstoDirichlet}
In this lemma, for any function $g:\N^d\mapsto \R$ we will write $g\left(\sum\limits_i n_i \overrightarrow{e_i}\right)$ instead of $g(n_1,\dots,n_d)$. Let $f:\N^d\mapsto \R$ be a function that satisfies:
\[
f(0)=1 \text{ and }\forall (n_1,\dots,n_d)\in \N^d,\ f\left(\sum\limits_{1\leq i \leq d} n_i\overrightarrow{e_i}\right) = \sum\limits_{1\leq j \leq d} f\left(\overrightarrow{e_j}+\sum\limits_{1\leq i \leq d} n_i\overrightarrow{e_i}\right),
\]
and such that there exists functions $(h_i)_{1\leq i \leq d}$ from $\N$ to $\R$ and $\tilde{h}:\N \mapsto \R^*$ that satisfy:
\[
\begin{aligned}
&f\left(\sum\limits_{1\leq i \leq d} n_i\overrightarrow{e_i}\right)=\frac{\prod\limits_{1\leq i \leq d}h_i(n_i)}{\tilde{h}\left(\sum\limits_{1\leq i \leq d} n_i\right)},\\
\text{and }&\forall i,\ h_i(0)=1 \text{ and } h_i(1)\not=0,\\
\text{and }&\tilde{h}(0)=1.
\end{aligned}
\]
Then, either $\tilde{h}(2)\not=\tilde{h}(1)^2$ and there exists constants $(\beta_i)_{1\leq i \leq d}\in \R$ and $\gamma\not=0$ such that:
\[
\begin{aligned}
&f\left(\sum\limits_{1\leq i \leq d} n_i\overrightarrow{e_i}\right)=\frac{\Gamma\left(\sum\limits_{1\leq i \leq d} \beta_i\right)}{\Gamma\left(\sum\limits_{1\leq i \leq d} n_i+\beta_i\right)}\prod\limits_{1\leq i \leq d}\frac{\Gamma(n_i+\beta_i)}{\Gamma(\beta_i)},\\
\text{and }&\forall i\in [|1,\dots,d|],\forall n\in \N, h_i(n)=\gamma^n H(\beta_i,n),\\
\text{and }&\forall n\in \N, \tilde{h}(n)=\gamma^n H\left(\sum\limits_{1\leq i \leq d} \beta_i,n\right),
\end{aligned}
\]
where $H(a,n)=\prod\limits_{0\leq i \leq n-1}(a+i)$.\\
Or $\tilde{h}(2)=\tilde{h}(1)^2$ and there exists constants $(c_i)_{1\leq i \leq d}$ and $\gamma \not=0$ such that:
\[
\begin{aligned}
&f\left(\sum\limits_{1\leq i \leq d} n_i\overrightarrow{e_i}\right)=\prod\limits_{1 \leq i \leq d}(c_i)^{n_i},\\
\text{and }&\forall i\in [|1,\dots,d|],\forall n\in \N, h_i(n)=(\gamma c_i)^n,\\
\text{and }&\forall n\in \N, \tilde{h}(n)=\gamma^n.
\end{aligned}
\]
\end{lem}
\begin{proof}
First, we want to show that proving the result when $\tilde{h}(1)=1$ is enough. Indeed if we look at the functions:
\[
g_i(n)=\frac{h_i(n)}{\tilde{h}(1)^n} \text{ and } 
\tilde{g}(n)=\frac{\tilde{h}(n)}{\tilde{h}(1)^n},
\]
We still have:
\[
f\left(\sum\limits_{1\leq i \leq d} n_i\overrightarrow{e_i}\right)=\frac{\prod\limits_{1\leq i \leq d}g_i(n_i)}{\tilde{g}\left(\sum\limits_{1\leq i \leq d} n_i\right)}.
\]
We also have that $\tilde{h}(2)=\tilde{h}(1)^2$ if and only if $\tilde{g}(2)=\tilde{g}(1)^2$. We also have $\tilde{g}(1)=1$. Therefore looking at the case $\tilde{h}(1)=1$ is enough so we will only look at that case. \\
-If $\tilde{h}(2)\not=\tilde{h}(1)^2$. Let $\beta:=\frac{\tilde{h}(1)^2}{\tilde{h}(2)-\tilde{h}(1)^2}$ so that $\frac{\beta}{1+\beta}=\frac{\tilde{h}(1)^2}{\tilde{h}(2)}$. In particular, $\beta\not=0$. We will choose the following values for the $\beta_i$ and prove that they yield the desired result:
\[
\forall i, \beta_i:= f(\overrightarrow{e_i})\beta.
\]
Now we will prove by induction on $\sum\limits_{1\leq i \leq d}n_i $ that these values of $\beta_i$ are correct. To avoid problems of definition of the gamma function, we will use the following function $H:\R\times\N\mapsto \N$ defined by:
\[
\begin{aligned}
&H(a,0)=1,\\
\text{and }&H(a,n)=\prod\limits_{0\leq i \leq n-1} (a+i).
\end{aligned}
\]
When the gamma function is well-defined, we have the equality:
\[
H(a,n)=\frac{\Gamma(a+n)}{\Gamma(a)}.
\]
Now, we want to show by induction that:
\[
h_i(n)=\frac{1}{\beta^n}H(\beta_i,n)\text{ and }\tilde{h}(n)=\frac{1}{\beta^n}H(\beta,n).
\] 
The result is obviously true for $h_i(0)$, $\tilde{h}(0)$ and $\tilde{h}(1)$. We also have, by definition of $(\beta_i)_{1\leq i\leq d}$:
\[
h_i(1)=\frac{h_i(1)}{\tilde{h}(1)}=f(\overrightarrow{e_i})=\frac{\beta_i}{\beta}=\frac{H(\beta_i,1)}{\beta}.
\]
We also have, by definition of $\beta$: 
\[
\tilde{h}(2)=\frac{1+\beta}{\beta}=\frac{\beta(1+\beta)}{\beta^2},
\]
so we have the desired result. Now we have to look at $f(\overrightarrow{2e_i})$:
\[
\begin{aligned}
f(\overrightarrow{2e_i})&=f(\overrightarrow{e_i})-\sum\limits_{j \not = i} f(\overrightarrow{e_i}+\overrightarrow{e_j})\\
&=\frac{\beta_i}{\beta}-\sum\limits_{j \not = i}\frac{\beta_i \beta_j}{\beta (1+\beta)}\\
&=\frac{\beta_i}{\beta}-\frac{\beta_i}{(1+\beta)}\sum\limits_{j \not = i}\frac{\beta_j}{\beta}\\
&=\frac{\beta_i}{\beta}-\frac{\beta_i}{(1+\beta)}\frac{\beta-\beta_i}{\beta}\\
&=\frac{\beta_i (1+\beta_i)}{\beta(1+\beta)}.
\end{aligned}
\]
So: 
\[
h_i(2)=\tilde{h}(2)f(\overrightarrow{2e_i})=\frac{1+\beta}{\beta}\frac{\beta_i (1+\beta_i)}{\beta(1+\beta)}=\frac{\beta_i (1+\beta_i)}{\beta^2}.
\]
We have the desired result, now we can prove the result by induction on $n$. We have already proved it if $n \leq 2$. Now we assume it is proved for $n\leq N$ and we will prove it for $N+1$.
First we use the equality:
\[
f((N-1)\overrightarrow{e_1}+\overrightarrow{e_2})=\sum\limits_{1\leq i \leq d} f((N-1)\overrightarrow{e_1}+\overrightarrow{e_2}+\overrightarrow{e_i}).
\]
So:
\[
\begin{aligned}
\frac{h_1(N-1)h_2(1)}{\tilde{h}(N)}&=\frac{1}{\tilde{h}(N+1)}\left(h_1(N)h_2(1)+h_1(N-1)h_2(2)+\sum\limits_{3\leq i \leq d}h_1(N-1)h_2(1)h_i(1)\right)\\
&=\frac{1}{\beta^{N+1}\tilde{h}(N+1)}H(\beta_1,N-1)\beta_2\left(\beta_1+N-1+\beta_2+1+\sum\limits_{3\leq i \leq d} \beta_i\right)\\
&=\frac{1}{\beta^{N+1}\tilde{h}(N+1)}H(\beta_1,N-1)\beta_2(\beta+N).
\end{aligned}
\]
So we get:
\[
\frac{H(\beta_1,N-1)}{\beta^{N-1}\tilde{h}(N)}=\frac{H(\beta_1,N-1)(\beta+N)}{\beta^{N}\tilde{h}(N+1)}.
\]
Therefore:
\[
\tilde{h}(N+1)H(\beta_1,N-1)=\frac{1}{\beta}\tilde{h}(n)H(\beta_1,N-1)(\beta+N).
\]if we replace $\beta_1$ by any $\beta_i$. So either there exists $i$ such that $H(\beta_i,N-1)\not=0$ and we get:
\[
\tilde{h}(N+1)=\frac{H(\beta,N)}{\beta^{N+1}}
\]
or $\forall 1\leq i \leq d, H(\beta_i,N-1)=0$, which means that:
\[
\forall i, -\beta_i\in\N \text{ and } \beta_i\geq -(N-2).
\]
In the latter case, we have:
\[
f\left(\sum\limits_{1\leq i \leq d}-\beta_i\overrightarrow{e_i}\right)=\frac{\prod\limits_{1\leq i \leq d}H(\beta_i,-\beta_i)}{\beta^{-\beta}\tilde{h}(-\beta)}\not=0.
\]
However, we also have:
\[
f\left(\sum\limits_{1\leq i \leq d}-\beta_i\overrightarrow{e_i}\right)=\sum\limits_{1\leq j \leq d}f\left(\overrightarrow{e_j} +\sum\limits_{1\leq i \leq d}-\beta_i\overrightarrow{e_i}\right)=0
\]
so it is not possible.\\
Finally:
\[
\begin{aligned}
h_i(N+1)&=f((N+1)\overrightarrow{e_i})\tilde{h}(N+1)\\
&=\tilde{h}(N+1)\left(f(N\overrightarrow{e_i})-\sum\limits_{j\not=i}f(N\overrightarrow{e_i}+\overrightarrow{e_j})\right)\\
&=\frac{\tilde{h}(N+1)}{\tilde{h}(N)}h_i(N)-\sum\limits_{j\not=i}h_i(N)h_j(1)\\
&=\frac{\tilde{h}(N+1)}{\tilde{h}(N)}h_i(N)-\sum\limits_{j\not=i}h_i(N)\frac{\beta_j}{\beta}\\
&=h_i(N)\left(\frac{\beta+N}{\beta}-\sum\limits_{j\not=i}\frac{\beta_j}{\beta}\right)\\
&=h_i(N)\frac{\beta_i+N}{\beta}\\
&=\frac{H(\beta_i,N+1)}{\beta^{N+1}}.
\end{aligned}
\]
So we have the result we want.\\
-If $\tilde{h}(2)=\tilde{h}(1)^2=1$, we want to show that $\tilde{h}(n)=1$ and $h_i(n)=f(\overrightarrow{e_i})^n$. We note $C_i:=f(\overrightarrow{e_i})$ to simplify notations. We have the following equality:
\[
\sum\limits_{1\leq i \leq d} C_i=\sum\limits_{1\leq i \leq d}f(\overrightarrow{e_i})=f(0)=1
\]
We want to show the result by induction on $n$, it is obvious for $n=0$ and $n=1$. Now, we assume the result is proved for $n\leq N$, we want to prove it for $N+1$.
First we use the equality:
\[
f((N-1)\overrightarrow{e_1}+\overrightarrow{e_2})=\sum\limits_{1\leq i \leq d} f((N-1)\overrightarrow{e_1}+\overrightarrow{e_2}+\overrightarrow{e_i}).
\]
We get:
\[
\begin{aligned}
C_1^{N-1}C_2&=\frac{1}{\tilde{h}(N+1)}\left(C_1^{N}C_2+C_1^{N-1}C_2^2+\sum\limits_{3\leq i \leq d}C_1^{N-1}C_2C_i\right)\\
&=\frac{C_1^{N-1}C_2}{\tilde{h}(N+1)}\left(C_1+C_2+\sum\limits_{3\leq i \leq d} C_i\right)\\
&=\frac{C_1^{N-1}C_2}{\tilde{h}(N+1)},
\end{aligned}
\]
since $C_1^{N-1}C_2\not=0$, we get $\tilde{h}(N+1)=1$.\\
Finally:
\[
\begin{aligned}
h_i(N+1)&=f((N+1)\overrightarrow{e_i})\\
&=f(N\overrightarrow{e_i})-\sum\limits_{j\not=i}f(N\overrightarrow{e_i}+\overrightarrow{e_j})\\
&=C_i^N-\sum\limits_{j\not=i}C_i^NC_j\\
&=C_i^N\left(1-\sum\limits_{j\not=i}C_j\right)\\
&=C_i^{N+1}.
\end{aligned}
\]
So, we have the result we wanted.
\end{proof}
Now we have all we need to prove the theorem.
\begin{theo2}
Let $(V,E)$ be a finite directed graph, with no multiple edges, 2-connected and such that its reversed graph is also 2-connected. Then the only RWREs with a non-deterministic environment on this graph with independent transition probability such that for every edge $(x,y)$, we have $\E\left(p_{\omega}(x,y)\right)>0$ and the reversed walk also has independent transition probabilities is a RWRE where the transition probabilities are independent Dirichlet random variables with null divergence.
\end{theo2}
\begin{proof}
Let $(w(x,y))_{(x,y)\in E}$ be random transition probabilities and $(\check{w}(x,y))_{(x,y)\in\tilde{E}}$ the transition probabilities of the reversed environment such that the transition probabilities are independent at each site both for the environment and the reversed environment. We define the moment functions $f$ and $\check{f}$ by:
\[
\begin{aligned}
\forall x\in V, f_x\left(\sum\limits_{y\in V_x} n_{xy}\overrightarrow{xy} \right)&=\E\left(\prod\limits_{y\in E_x} (w(x,y))^{n_{xy}}\right) \text{ and }\\
\forall x\in V, \check{f}_x\left(\sum\limits_{y\in V^x} n_{yx}\overrightarrow{xy} \right)&=\E\left(\prod\limits_{y\in E^x} (\check{w}(x,y))^{n_{yx}}\right).
\end{aligned}
\]
These two moments functions are compatible because the transition-probabilities of the time-reversed random walk are defined by $\check{w}(y,x)\pi_y=w(x,y)\pi_x$ where $(\pi_x)_{x\in V}$ is the stationary law. Therefore, if $N: E \mapsto \N$ is of null divergence then:
\[
\begin{aligned}
\prod\limits_{v\in V}\check{f}_v\left(\sum\limits_{u\in V^v} N((u,v))\overrightarrow{vu}\right)&=
\prod\limits_{v\in V}\E\left(\prod\limits_{u\in V^v}(\check{w}(v,u))^{N((u,v))}\right)\\
&=\E\left(\prod\limits_{v\in V}\prod\limits_{u\in V^v}(\check{w}(v,u))^{N((u,v))} \right)\\
&=\E\left(\prod\limits_{v\in V}\prod\limits_{u\in V^v}\left(w(u,v)\frac{\pi_u}{\pi_v}\right)^{N((u,v))} \right)\\
&=\E\left(\prod\limits_{v\in V}\prod\limits_{u\in V^v}(w(u,v))^{N((u,v))}\prod\limits_{v\in V} (\pi_v)^{\sum\limits_{u\in E_v} N((v,u))-\sum\limits_{u\in V^v} N((u,v))} \right)\\
&=\E\left(\prod\limits_{v\in V}\prod\limits_{u\in V^v}(w(u,v))^{N((u,v))}\right)\\
&=\prod\limits_{v\in V}\E\left(\prod\limits_{u\in V_v}(w(v,u))^{N((v,u))}\right)\\
&=\prod\limits_{v\in V}f_v\left(\sum\limits_{u\in V_v} N((v,u))\overrightarrow{vu}\right).
\end{aligned}
\]
Now we can apply the result of lemma \ref{lem:hfunctions} which gives the existence of functions $\tilde{h}_x:\N\mapsto (0,\infty)$ for every $x\in V$ and functions $h_e: \N \mapsto (0,\infty)$ for every $e\in E$ such that:
\[
\begin{aligned}
&\forall x\in V, \forall y\in V_x, \forall n\in \N, f_x(n\overrightarrow{xy})=\frac{h_{(x,y)}(n)}{\tilde{h}_x(n)} \text{ and} \\
&\forall x\in V, \forall y\in V^x, \forall n\in \N, \check{f}_x(n\overrightarrow{xy})=\frac{h_{(y,x)}(n)}{\tilde{h}_x(n)}.
\end{aligned}
\]
Now we can consider the moment functions $g$ and $\check{g}$ defined by:
\[
\begin{aligned}
&\forall x\in V, g_x\left(\sum\limits_{y\in V_x} N((x,y))\overrightarrow{xy}\right)=f_x\left(\sum\limits_{y\in V_x} N((x,y))\overrightarrow{xy}\right)\frac{\tilde{h}_x\left(\sum\limits_{y\in V_x} N((x,y))\right)}{\prod\limits_{y\in V_x} h_{(x,y)}(N((x,y)))} \text{ and} \\
&\forall x\in V, \check{g}_x\left(\sum\limits_{y\in V^x} N((y,x))\overrightarrow{xy}\right)=\check{f}_x\left(\sum\limits_{y\in V^x} N((y,x))\overrightarrow{xy}\right)\frac{\tilde{h}_x\left(\sum\limits_{y\in V^x} N((y,x))\right)}{\prod\limits_{y\in V^x} h_{(y,x)}(N((y,x)))}.
\end{aligned}
\]
The moment functions $g$ and $\check{g}$ are compatible that satisfy the hypotheses of lemma \ref{lem:deltacorrection} so there exists a function $\Delta:\N\mapsto (0,\infty)$ such that:
\[
\begin{aligned}
&\forall x\in V, g_x\left(\sum\limits_{y\in V_x} N((x,y))\overrightarrow{xy}\right)=\frac{\prod\limits_{y\in V_x} \Delta(N((x,y)))}{\Delta\left(\sum\limits_{y\in V_x}N((x,y))\right)}\text{ and} \\
&\forall x\in V, \check{g}_x\left(\sum\limits_{y\in V^x} N((y,x))\overrightarrow{xy}\right)=\frac{\prod\limits_{y\in V^x} \Delta(N((y,x)))}{\Delta\left(\sum\limits_{y\in V^x}N((y,x))\right)}.
\end{aligned}
\]
We define the following functions: 
\[
\begin{aligned}
&\forall e\in E, h^{\prime}_e(n):=\Delta(n)h_e(n)\\
\text{and }&\forall x\in V, \tilde{h}^{\prime}_x(n):=\Delta(n)\tilde{h}_x(n).
\end{aligned}
\]
We have:
\[
\begin{aligned}
&\forall x\in V, f_x\left(\sum\limits_{y\in V_x} N((x,y))\overrightarrow{xy}\right)=\frac{\prod\limits_{y\in V_x} h^{\prime}_{(x,y)}(N((x,y)))}{\tilde{h}^{\prime}_x\left(\sum\limits_{y\in V_x}N((x,y)) \text{ and}\right)} \\
&\forall x\in V, \check{f}_x\left(\sum\limits_{y\in V^x} N((y,x))\overrightarrow{xy}\right)=\frac{\prod\limits_{y\in V^x} h^{\prime}_{(y,x)}(N((y,x)))}{\tilde{h}^{\prime}_x\left(\sum\limits_{y\in V_x}N((y,x))\right)}.
\end{aligned}
\]
Now, according to lemma \ref{lem:momentstoDirichlet}, we get that for every $x\in V$, $f_x$ is either the moments of a Dirichlet distribution or the moments of a deterministic distribution, the same is true for $\check{f}_x$. Either all of the $f_x$ and $\check{f}_x$ are moments of Dirichlet distribution, or at least one of them is deterministic in which case we can assume that it is $f_x$ for some $x \in V$ that will be fixed for the rest of the proof. We want to prove that in the later case, all the probability transitions are deterministic. For any $x\in V$, if either $f_x$ or $\check{f}_x$ is deterministic then according to lemma \ref{lem:momentstoDirichlet} there is a $\gamma_x\not=0$ such that $\tilde{h}^{\prime}_x(n)=\gamma_x^n$ and therefore, still according to lemma \ref{lem:momentstoDirichlet} if $\tilde{h}^{\prime}_x(n)=\gamma_x^n$ then we are in the case where $\tilde{h}^{\prime}_x(2)=\tilde{h}^{\prime}_x(1)^2$ and both $f_x$ and $\check{f}_x$ are the moments of deterministic transitions probabilities. Now let $(x,y)\in E$ be an edge, if $f_x$ is the moment of deterministic transitions probabilities then, once again by lemma \ref{lem:momentstoDirichlet}, there is a $\gamma_{(x,y)}\not=0$ such that $h^{\prime}_{(x,y)}(n)=\gamma_{(x,y)}^n$. Now we get:
\[
\check{f}_y\left(n\overrightarrow{yx}\right)=\frac{h^{\prime}_{(x,y)}(n)}{\tilde{h}^{\prime}_y\left(n\right)}.
\]
According to lemma \ref{lem:momentstoDirichlet} the only way to have $h^{\prime}_{(x,y)}(n)=\gamma_{(x,y)}^n$ is that $\check{f}_y$ and $f_y$ are the moments of deterministic transition probabilities. Now since the graph is connected, we get that for all vertices $z\in V$, the function $f_s$ is the moments of deterministic transition probabilities. \\
If for every $x\in V$, $f_x$ is the moments of a Dirichlet distribution then we want to prove that we have null divergence. According to lemma \ref{lem:momentstoDirichlet} there exists $(\beta_e)_{e\in E}$ such that for every $x\in V$, $f_x$ is the moments of a Dirichlet distribution of parameters $(\beta_e)_{e\in E_x}$ and $\check{f}_x$ is the moments of a Dirichlet distribution of parameters $(\beta_e)_{e\in E^x}$. There exists $(\gamma_x)_{x\in V},(\beta_x)_{x\in V}$ such that:
\[
\forall x \in V,\forall n\in\N, \tilde{h}^{\prime}_x(n)=\gamma_x^n\frac{\Gamma(\beta_x + n)}{\Gamma(\beta_x)}.
\]
Since for every $x\in V$, $\sum\limits_{y\in V_x}\omega(x,y)=1$ almost surely, we have:
\[
\forall x\in V \sum\limits_{y\in V_x}\frac{h^{\prime}_{(x,y)}(1)}{\tilde{h}^{\prime}_x(1)}=1.
\]
This means, according to lemma \ref{lem:momentstoDirichlet} that:
\[
\forall x\in V, \sum\limits_{y\in V_x}\beta_{(x,y)}=\beta_x.
\]
The same way we get, by looking at the reversed walk:
\[
\forall x\in V, \sum\limits_{y\in V^x}\beta_{(y,x)}=\beta_x.
\]
This means that the parameters of the Dirichlet distributions have null divergence.
\end{proof}

\bibliographystyle{abbrv}
\bibliography{biblio}
\end{document}